%% file: bers.tex
\documentclass[12pt]{amsart} %or amsproc, or amsbook

\input{Preamble}

\input{Definitions}

\begin{document}

\title[Uryson width and pants decompositions of hyperbolic surfaces]{Uryson width and pants decompositions of hyperbolic surfaces}

\author{Gregory R. Chambers}
\email{gchambers@rice.edu}
\address{Department of Mathematics, Rice University, Houston, Texas, USA}
\date{\today}
\begin{abstract}
    Suppose that $M$ is a hyperbolic closed orientable surface of genus $g$ and with $n$ cusps.
    Then we can find a pants decomposition of $M$ composed of simple closed geodesics so that each curve is contained in a ball of diameter at most $C \sqrt{g + n}$, where $C$ is a universal constant.
\end{abstract}
\maketitle

%%%%%%%%%%%%%%%%%%%%%%%%%%%%%%%%%%%%%%%%%%%%%%%%%%%%%%%%%%%%%
%%%%%%%%%%%%%%%%%%%%%%%%%%%%%%%%%%%%%%%%%%%%%%%%%%%%%%%%%%%%%
\section{Introduction}
\label{sec:intro}
%%%%%%%%%%%%%%%%%%%%%%%%%%%%%%%%%%%%%%%%%%%%%%%%%%%%%%%%%%%%%
%%%%%%%%%%%%%%%%%%%%%%%%%%%%%%%%%%%%%%%%%%%%%%%%%%%%%%%%%%%%%

In this article, we will examine pants decompositions of hyperbolic surfaces.  In particular, suppose that $M$ is a hyperbolic
surface with genus $g$ and with $n$ cusps.  A pants decomposition of such a surface is a finite sequence $\gamma_1, \dots, \gamma_k$ of simple closed smooth curves which are pairwise disjoint, and so that if we remove their images from $M$, the remainder is a finite union of thrice punctured hyperbolic spheres.  If $g + 2n < 3$, then no decomposition exists; we will assume
that $g + 2n \geq 3$.  

In general there are possible ways to decompose such a hyperbolic surface; we will be interested in the lengths of
the curves $\gamma_1, \dots, \gamma_k$.  In \cite{Bers1} and \cite{Bers2}, Bers showed that in $M$ is closed, 
there is a choice so that these lengths are all bounded by a constant that depends only on the genus; these bounds
are called Bers' constants.

In \cite{Buser1}, \cite{Buser2}, and \cite{Buser3}, Buser studied optimal bounds on these constants, producing linear
upper bounds and square root lower bounds (the square root of $g + n$ if the surface has cusps).  He made the following conjecture:

\begin{conj}[Buser]
    \label{conj:buser}
    Suppose that $M$ is a hyperbolic surface with genus $g$ and $n$ cusps.  Then it has a pants decomposition in which each curve has length at most $C \sqrt{g + n}$ for some universal constant $C$.
\end{conj}

A more in depth discussion of the background of his conjecture can be found in the introduction of \cite{BP}, in which
Balacheff and Parlier prove Conjecture~\ref*{conj:buser} if $g = 0$. In this article, we prove the following:
\begin{thm}
    \label{thm:main}
    Suppose that $M$ is a hyperbolic manifold with genus $g$ and $n$ cusps.  Then $M$ has a pants decomposition so that each curve is a simple closed geodesic, and each has diameter at most $C \sqrt{g + n}$.
\end{thm}

A higher-dimensional version of this theorem, sweeping out closed Riemannian manifolds using $1$-dimensional cycles,
was proved by Nabutovsky, Rotman, and Sabourau in \cite{NRS}.

The organization of the remainder of the article is as follows. In Section~\ref*{sec:proof}, we prove Theorem~\ref{thm:main}, leaving several main components of the proof for later discussion.  These components involve
bounds on the Uryson width of a closed Riemannian manifold, which will be discussed in Section~\ref*{sec:uryson}, and 
a curve shortening process developed by Hass and Scott, which will be discussed in Section~\ref*{sec:short}.

We close this section with a few remarks.  Throughout this article, we will use the notation $A \lesssim B$ to mean
that there is some universal constant $C$ so that $A \leq C B$.  If $C$ depends on the dimension $n$ only, then we will write $A \lesssim_n B$ to mean $A \leq C(n) B$.  We define $\gtrsim$ and $\gtrsim_n$ analogously.  We define $A \approx B$ to mean
$A \lesssim B$ and $B \lesssim A$, and $A \approx_n B$ to mean $A \lesssim_n B$ and $B \lesssim_n A$.

We will also state the following proposition containing several standard facts about hyperbolic surfaces; these will be useful later on. Since these results are standard, we omit their proofs:
\begin{prop}
    \label{prop:hyperbolic}
    Suppose that $M$ is a hyperbolic surface of genus $g$ and with $n$ cusps.  Then the following are true:
    \begin{enumerate}
        \item The area of $M$ is $\lesssim \sqrt{g + n}$.
        \item For every cusp of $M$, there is an open subset $U$ of $M$ which is isometric to
            $D = \{ z \in \mathbb{R}^2 : 0 < |z| < 1\}$ with a metric $G$ so that:
            \begin{enumerate}
                \item The lengths of the circle $C_{\rho} = \{ z \in \mathbb{R}^2 : |z| = \rho \}$ goes to $0$ as $\rho \rightarrow 0^+$ (with $\rho \in (0,1)$).
                \item For every $\rho \in (0,1)$, if $\{x_i\}$ and $\{y_i\}$ are sequences of points in $D$
                    with $|x_i| > \rho$ and $|y_i| \rightarrow 0^+$, then
                    the distance from $x_i$ to $y_i$ goes to $\infty$.
            \end{enumerate}
    \end{enumerate}
\end{prop}

\noindent {\bf Acknowledgments} The author would like to
thank Larry Guth for first introducing him to this problem.  He would also like to thank Alexander Nabutovsky, Regina Rotman, Yevgeny Liokumovich, Robert Young, Arnaud de Mesmay, and St\'{e}phane Sabourau for many discussions about this problem and surrounding literature.
The author would also like to thank Alan Reid for helpful comments on the initial draft of this article.  Lastly, the author would like to thank his mother, Catherine Chambers, for
taking care of his son William while this article was being written.  The author would also like to thank Maxime Fortier Bourque and Bram Petri for pointing out errors in a previous version of this article.  The research of the author was partially supported by NSF Grant DMS-1906543.

%%%%%%%%%%%%%%%%%%%%%%%%%%%%%%%%%%%%%%%%%%%%%%%%%%%%%%%%%%%%%%%%%%%%%%%%%%%%%%
%%%%%%%%%%%%%%%%%%%%%%%%%%%%%%%%%%%%%%%%%%%%%%%%%%%%%%%%%%%%%%%%%%%%%%%%%%%%%%
\section{Proof of Theorem~\ref*{thm:main}}
\label{sec:proof}
%%%%%%%%%%%%%%%%%%%%%%%%%%%%%%%%%%%%%%%%%%%%%%%%%%%%%%%%%%%%%%%%%%%%%%%%%%%%%%
%%%%%%%%%%%%%%%%%%%%%%%%%%%%%%%%%%%%%%%%%%%%%%%%%%%%%%%%%%%%%%%%%%%%%%%%%%%%%%

In this section, we prove Theorem~\ref*{thm:main}.  As mentioned in the introduction, if $n + 2g < 3$, then
no such pants decomposition exists.  We will deal with the $n + 2g \geq 3$ and $g = 1$ case at the end of this section.
If $n + 2 g \geq 3$ and $g \geq 2$, then we proceed with the following definition:

\begin{defn}
    \label{defn:pseudo_pants}
    Suppose that $M$ is a hyperbolic surface of genus $g$ and $n$ cusps.
    A pseudo pants decomposition is a finite sequence $\gamma_1, \dots, \gamma_k$
    of simple closed pairwise disjoint curves so that $M$ with $\gamma_1, \dots, \gamma_k$ removed consists of
    a finite union of spheres, each of which has three punctures formed from removing
    $\gamma_1, \dots, \gamma_k$, and between $0$ and $n$ of the original $n$ cusps.
\end{defn}

The first part of the proof will involve finding a ``good" pseudo pants decomposition.
Such a pants decomposition is one in which the ambient diameter of each curve
is $\lesssim \sqrt{g + n}$.  This is stated in the following proposition:

\begin{prop}
    \label{prop:exist_pseudo}
    Suppose that $M$ is a hyperbolic surface with genus $g$ and $n$ cusps.  If $g \geq 2$, then there exists a pseudo pants decomposition so that the ambient diameter
    of each curve is $\lesssim \sqrt{g + n}$.
\end{prop}

To prove this proposition, we will use bounds on the Uryson width.  We will discuss
the definition of Uryson width and some background in Section~\ref*{sec:uryson}.  The main
result that we will need, first proved in \cite{G} by Guth in 2011, is as follows:

\begin{thm}
    \label{thm:guth}
    If $M$ is a closed Riemannian $n$-manifold, then there is an $(n-1)$-dimensional
    simplicial complex $\Gamma$ and a continuous function $f : M \rightarrow \Gamma$
    so that, for every $P \in \Gamma$,  $f^{-1}(p) $
    has ambient diameter $\lesssim_n \textrm{Vol}(M)^{1/n}.$
\end{thm}

The idea for the proof of Proposition~\ref*{prop:exist_pseudo} will be to first
identify short curves around each cusp, to cut the cusp out using this curve,
and then to fill the curve with a portion of a small sphere.  After smoothing
out the result, we use Theorem~\ref*{thm:guth}, then map the simplicial complex to $\mathbb{R}$, and then finally approximate the resulting function with a Morse function.

For a suitably small $C^0$ approximation, we will show that we can make the same conclusion about the preimages of the Morse function, with a slightly worse constant.  We use the Morse function to find a pants decomposition of the aforementioned closed Riemannian manifold.  After arguing that these curves lie outside the pieces that we glued in, we obtain our pseudo pants decomposition.  Since $\textrm{Area}(M) \approx \sqrt{g + n}$, this yields
a good bound on the diameter.

The next step is to take this pseudo pants decomposition, and to replace the curves with simple closed geodesics:

\begin{prop}
    \label{prop:pants_geodesics}
    Suppose that $M$ is a hyperbolic surface and $\gamma_1, \dots, \gamma_k$ is
    a pseudo pants decomposition for $M$ so that the ambient diameter of each
    $\gamma_i$ is $\leq D$.  Then there exists a pseudo pants decomposition
    $\tilde{\gamma}_1, \dots, \tilde{\gamma}_k$ so that each curve is a simple closed geodesic
    and has ambient diameter $D$.
\end{prop}

The method that we employ to prove Proposition~\ref*{prop:pants_geodesics} is
a curve shortening process developed by Hass and Scott in \cite{HS}.  This process
resembles the Birkhoff curve shortening process; it involves replacing short segments of a given sequence of curves $\alpha_1, \dots, \alpha_m$ with geodesic segments, forming a new sequence of collections of curves
$\tilde{\alpha_1}, \dots, \tilde{\alpha}_m$.  This process has the important properties that if the original curves are simple and disjoint, then the new curves are also
simple and disjoint.  In addition, they are (respectively) homotopic to the original
curves.

Lastly, and critically, if we continue to repeat this procedure, then we obtain
$m$ sequences of curves; each sequence of curves converges to a closed geodesic.
If we begin with a pseudo pants decomposition, then these final closed geodesics
are simple and disjoint, and also form a pseudo pants decomposition.

Next, we argue that since this process replaces small segments with short geodesic arcs,
and since each of the original curves has ambient diameter $\lesssim \sqrt{g + n}$,
the new curves still have this diameter bound.  This uses the fact that the surface is hyperbolic.  Furthermore,
these curves are all closed geodesics, as desired. To complete the proof, we use the results of Balacheff and Parlier in \cite{BP}.
Their main theorem is a proof of Conjecture~\ref*{conj:buser} for $g = 0$:

\begin{thm}[Balacheff and Parlier]
    \label{thm:sphere}
    Suppose that $M$ is a hyperbolic sphere with $n$ cusps.  Then $M$ has a pants decomposition composed of curves each of length $\lesssim \sqrt{n}$.
\end{thm}

We would like to apply this theorem to each of the components of the manifold after we remove the $k$ simple closed geodesics that form the pseudo pants decomposition that we found above.
However, each component has $3 + m$ punctures, where $3$ punctures come from three curves
which are removed, and the $m$ punctures come from cusps $(0 \leq m \leq n)$.
In the same article, they prove a similar result for hyperbolic spheres with boundary geodesics.
This is stated in Lemma~8 of that article.

The approach is, for each geodesic boundary component $\gamma$, to glue a hyperbolic pair of pants to it with two cusps and one geodesic boundary component equal in length to $\gamma$.  This produces
a sphere with cusps, to which we can then apply Theorem~\ref*{thm:sphere}.  
Furthermore, Balacheff and Parlier show that we can force the original closed geodesics to be in this pants
decomposition at the expense of adding the sum of the lengths of all of the
geodesics to the bound.  We state this as follows;
the proof follows from Theorem~\ref*{thm:sphere} and Lemma~8 in \cite{BP} as described:

\begin{thm}[Balacheff and Parlier]
    \label{thm:sphere_2}
    Suppose that $M$ is a hyperbolic sphere with $m$ cusps and $k$ boundary components,
    each of which is a geodesic of length at most $\ell$.  Then we can find a pants decomposition
    of $M$ so that each curve has length at most
        $$ \lesssim \sqrt{m} + k \ell. $$
\end{thm}

We then apply Theorem~\ref*{thm:sphere_2} to each of the components we found above; in this case
$k = 3$, $0 \leq m \leq n$, and $\ell \lesssim \sqrt{g + n}$.  Thus, each curve that we had added is itself contained in a ball of radius $\lesssim \sqrt{g + n}$.  We then apply
the approach of Hass and Scott described above to all of the curves (the original curves and the new ones) to complete the proof of Theorem~\ref*{thm:main}.

If $g = 1$, then we find the systole $\gamma$ of $M$.  This is the shortest
noncontractible curve; it is smooth and simple.  Gromov proved that it has length at most $\lesssim  \sqrt{\textrm{Area}(M)}$ in \cite{Gromov} (for hyperbolic surfaces we can actually do better, bounding
the systole by the logarithm of the genus times a constant). Since $\textrm{Area}(M) \approx \sqrt{g + n}$, if we choose this curve $\gamma$ as the single curve then remove it, then the result is a sphere with exactly two geodesic boundary components (each of length $\lesssim \sqrt{g + n}$ and $n$ cusps.  We then apply the above results of Balacheff and Parlier to complete the proof of this case.  This concludes the proof of Theorem~\ref*{thm:main}.

%%%%%%%%%%%%%%%%%%%%%%%%%%%%%%%%%%%%%%%%%%%%%%%%%%%%%%%%%%%%%%%%%
%%%%%%%%%%%%%%%%%%%%%%%%%%%%%%%%%%%%%%%%%%%%%%%%%%%%%%%%%%%%%%%%%
\section{Uryson Width and Proposition~\ref*{prop:exist_pseudo}}
\label{sec:uryson}
%%%%%%%%%%%%%%%%%%%%%%%%%%%%%%%%%%%%%%%%%%%%%%%%%%%%%%%%%%%%%%%%%
%%%%%%%%%%%%%%%%%%%%%%%%%%%%%%%%%%%%%%%%%%%%%%%%%%%%%%%%%%%%%%%%%

The purpose of this section is to prove Proposition~\ref*{prop:exist_pseudo}.  To do this, we will use bounds on the
Uryson $1$-width of a closed Riemannian surface.  We begin with the definition of Uryson width, a method of measuring how closely an $n$-manifold
resembles a $k$-dimensional simplicial complex first introduced by Uryson and popularized by Gromov in \cite{Gromov}:

\begin{defn}
    \label{defn:uryson_width}
    Suppose that $M$ is a closed Riemannian $n$-manifold, and $0 \leq k \leq n$.  We say that the Uryson $k$-width is bounded by $\rho$ if
    there is some $k$-dimensional simplicial complex $\mathcal{C}$ and a continuous function $f : M \rightarrow \mathcal{C}$
    so that the preimage of every point in $\mathcal{C}$ has ambient diameter $\leq \rho$.  The Uryson $k$-width is then
    the infimum over all such $\rho$.
\end{defn}

As stated in Theorem~\ref*{thm:guth}, Guth showed that if $M$ is a closed Riemannian $n$-manifold, then the Uryson $(n-1)$-width is bounded by
$\lesssim_n \textrm{Vol}(M)^{1/n}$; this is the main result that we will use to prove Proposition~\ref*{prop:exist_pseudo}, using $n = 2$.  Guth's proof was extended to Hausdorff Content (instead of volume) by Liokumovich, Lishak, Nabutovsky, and Rotman in \cite{LLNR}, and a new proof of both of these results was given in 2021 by Papasoglu \cite{P}.

We now will prove Proposition~\ref*{prop:pants_geodesics}.  Suppose that $M$ is a hyperbolic surface with genus $g \geq 2$ and $n$ cusps.
To apply Guth's result, we need to work with a closed Riemannian surface; to this end, we cut out the cusps.
Since the surface is hyperbolic, we can find $n$ curves $\alpha_1, \dots, \alpha_n$ which are smooth, simple, closed and disjoint so that each $\alpha_i$ encloses a punctured disc that contains exactly one cusp, and $\alpha_i$ and $\alpha_j$ enclose distinct
cusps if $i \neq j$.  For every $\epsilon > 0$, we can find such curves so that the sum of all of the lengths of $\alpha_i$ is
less than $\epsilon$.  We can do this because the surface is hyperbolic, and so we can use Proposition~\ref*{prop:hyperbolic}.  We will choose
$\epsilon > 0$ later.

We then delete the enclosed discs and cusps with respect to $\alpha_1, \dots, \alpha_n$.  For each one of these boundary components
$\alpha_i$, we cap it with a hemisphere whose equator has length equal to that of $\alpha_i$. We smooth out this gluing, and call the resulting 
manifold $\tilde{M}$.  We observe that, assuming that this smoothing is done on a sufficiently small scale,
$\textrm{Area}(\tilde{M}) \leq \textrm{Area}(M) + 100 n \epsilon^2 $.

We now denote $\tilde{M}$ as $N$, and apply Theorem~\ref*{thm:guth} to it.  This results in a $1$-complex $\Gamma$ along with a continuous
function $f : N \rightarrow \Gamma$ so that, for every $p \in \Gamma$, $f^{-1}(p)$ has diameter $\lesssim \sqrt{\textrm{Area}(N)}$.

We begin by embedding $\Gamma$ continuously into $\mathbb{R}^3$ so that, for every plane $H_z = \{ \{x,y,z\} : x,y \in \mathbb{R} \}$,
$H_z$ intersects the image of $\Gamma$ in a finite number of points.  Clearly, such an embedding exists.  For example,
we can embed the vertices of $\Gamma$ distinctly, and then join these vertices by smooth curves according to how they are joined in
$\Gamma$.  After a small perturbation, the result has the desired properties.  Note that we can assume that $\Gamma$ is finite since $N$ is closed.

We will fix such an embedding $h : \Gamma \rightarrow \mathbb{R}^3$.  Lastly, we define the projection $\pi : \mathbb{R}^3 \rightarrow \mathbb{R}$ by
$$ \pi(x,y,z) = z. $$  These maps are shown below:

\begin{equation*}
\begin{tikzcd}
N \arrow{r}{f} & \Gamma \arrow{r}{h} & \mathbb{R}^3 \arrow{r}{\pi} & \mathbb{R}
\end{tikzcd}
\end{equation*}

We can now state the main lemma that we will need:
\begin{lem}
    \label{lem:uryson_interval}
    Suppose that $f$, $h$, and $\pi$ are as above.  There is some $\rho > 0$ so that the following holds.  For every closed interval
    $I = [a,b]$ with $b \geq a$ and $b - a \leq \rho$, for every connected component $X_I$ of $$(h \circ \pi)^{-1}(I),$$
    the diameter of $f^{-1}(X_I)$ is $< 2 C \sqrt{\textrm{Area}(N)}$.
\end{lem}
\begin{proof}
    Since $N$ is closed, $(f \circ h \circ \pi)(N)$ is a closed interval $[\alpha,\beta]$.  Suppose that the conclusion of the lemma
    was not true.  We could then find a sequence of intervals $I_1, I_2, \dots$ along with connected components $X_{I_1}, X_{I_2}, \dots$
    of $\Gamma$ so that $X_{I_j}$ is a connected component of $(h \circ \pi)^{-1}(I_j)$ with the following properties:
    \begin{enumerate}
        \item The length of $I_j$ goes to $0$ as $j$ goes to $\infty$.
        \item The diameter of $f^{-1}(X_{I_j})$ is at least $2 C \sqrt{\textrm{Area}(N)}$ for every $j$.
    \end{enumerate}
    
    Let $p_j$ be the center point of $I_j$; since $X_{I_j}$ has positive diameter, $I_j$ must intersect $[\alpha,\beta]$, and so
    $p_j$ has a convergent subsequence which converges to some $p \in [\alpha,\beta]$.  For the remainder of the proof,
    we will assume that we have already passed to such a subsequence.
    
    For every $\rho > 0$, consider the interval $I_{\rho} = [p - \rho, p + \rho]$.  We claim that there is a connected component $X_{\rho}$ of
    $(h \circ \pi)^{-1}(I_{\rho})$ with diameter $\geq 2 C \sqrt{\textrm{Area}(N)}$. Fix a $\rho > 0$; there is some
    $j$ so that $I_j \subset I_{\rho}$, and so there is a connected component $X_{\rho}$ of $(h \circ \pi)^{-1}(I_\rho)$ with $X_j \subset X_{\rho}$.
    As a result,
        $$ \textrm{Diameter}(X_{\rho}) \geq \textrm{Diameter}(X_j) \geq 2 C \sqrt{\textrm{Area}(N)}. $$
        
    Since $\Gamma$ is compact, every $X_{\rho}$ is compact.  In addition, due to this compactness,
    we can find $\rho_1, \rho_2, \dots$ with 
        \begin{enumerate}
            \item $\rho_i \geq \rho_{i+1}$.
            \item $\rho_i \rightarrow 0^+$ as $i \rightarrow \infty$.
            \item $X_{\rho_{i+1}} \subset X_{\rho_i}$.
        \end{enumerate}
    Thus, $$\bigcap_{i=1}^\infty X_{\rho_i}$$ is compact and connected; it is a connected component of $f^{-1}(z)$,
    where $z = (h \circ \pi)^{-1}(p)$.  We denote this connected component by $X_0$.
    
    From the above, for each $\rho_i$, there are $x_{\rho_i}, y_{\rho_i} \in X_{\rho_i}$ so that the distance from $x_{\rho_i}$ to $y_{\rho_i}$ is
    $ \geq 2 C \sqrt{\textrm{Area}(N)}$ (this uses the fact that $X_{\rho_i}$ is compact).  Thus, passing to a subsequence twice,
    $x_{\rho_i} \rightarrow x$ and $y_{\rho_i} \rightarrow y$ as $i \rightarrow \infty$, and
    \begin{enumerate}
        \item $x \in X_0$ and $y \in X_0$.
        \item The distance from $x$ to $y$ is at least $2 C \sqrt{\textrm{Area}(N)}$.
    \end{enumerate}
    However, this means that the diameter of $X_0$ is at least $2 C \sqrt{\textrm{Area}(N)}$, which is a contradiction, completing the proof.
\end{proof}

We now continue with the proof of Proposition~\ref*{prop:exist_pseudo}.  We use Lemma~\ref*{lem:uryson_interval} to produce a pants decomposition
of $N$.   We begin with the functions $f$, $h$, and $\pi$ defined in the proof of Lemma~\ref*{lem:uryson_interval}, and let $\rho > 0$ be as in its conclusion.
$f \circ h \circ \pi$ is a continuous function from $N$ to $\mathbb{R}$.  We can find a Morse function $\tilde{f}$ from $N$ to $\mathbb{R}$ so that, for every $p \in N$,
    $$ | f(p) - \tilde{f}(p) | \leq \rho / 10. $$
This is because we can approximate (in $C^0$) every continuous function on a compact manifold with a smooth function, and every such smooth function
can be approximated (in $C^\infty$) by a Morse function.
        
Consider now $\gamma$, a connected component of $\tilde{f}^{-1}(q)$.  If $x,y \in \gamma$, then $| f(x) - \tilde{f}(x) | \leq \rho / 10$
and $| f(y) - \tilde{f}(y) | \leq \rho / 10$.  Since $\tilde{f}(x) = \tilde{f}(y) = q$, we have $| f(x) - f(y) | \leq \rho / 5$.
Hence, if we let $a = (f \circ h \circ \pi)(x)$, then $\gamma$ is entirely contained in $(f \circ h \circ \pi)^{-1}(I)$, where
    $$ I = [ a - \rho / 5, a + \rho / 5], $$
and so is contained in the preimage of $f$ of a connected component of $(h \circ \pi)^{-1}(I)$.  By Lemma~\ref*{lem:uryson_interval},
this means that $\gamma$ has diameter at most $2 C \sqrt{\textrm{Area}(N)}$.  Once we have this Morse function, it is straightforward
to produce our our pants decomposition of $N$.  To do this, if $F$ is our Morse function in which $s_1, \dots, s_k$ are the singular 
points, then we choose $2k$ points $s_1 - \kappa, s_1 + \kappa, \dots, s_k - \kappa, s_k + \kappa$ where $\kappa > 0$ is so small that
$[s_i - \kappa,s_i + \kappa] \cap [s_j - \kappa, s_j + \kappa] = \emptyset$ for every $i \neq j$.  We then cut $N$ along
    $$ f^{-1}(s_1 - \kappa), f^{-1}(s_1 + \kappa), \dots, f^{-1}(s_k - \kappa), f^{-1}(s_k + \kappa). $$
Each of these is a collection of smooth curves, each of which has ambient diameter $\lesssim \sqrt{\textrm{Area}(N)}$.  After we remove these curves,
we are left with a set of one, two, and three punctured spheres.  Since the genus of $N$ is at least $2$, there is at least one thrice punctured sphere.

We reglue the one punctured spheres along its single boundary component;
this results in a set of two and three punctures spheres.  For every two punctured sphere, we glue it back along one of its boundary components.
The result is a collection of thrice punctured spheres; this is our pants decomposition.

We now choose $\epsilon > 0$ so small so that the following two properties are true:
\begin{enumerate}
    \item $\textrm{Area}(N) \leq 2 \textrm{Area}(M)$
    \item For each $\alpha_i$, there is a simple closed smooth curve $\tilde{\alpha}_i$ so that:
        \begin{enumerate}
            \item $\tilde{\alpha}_i$ bounds a punctured disc which contains both $\alpha_i$ and its cusp.
            \item For any curve $\gamma$ which goes from $\alpha_i$ to $\tilde{\alpha}_i$, $\gamma$ has length at $ \geq 10 C \sqrt{\textrm{Area}(M)}$.
        \end{enumerate}
\end{enumerate}

We can find $\epsilon > 0$ so that the first inequality is satisfied because of the relationship between the areas of $M$ and $N$ described above.  For the second, we Proposition~\ref*{prop:hyperbolic} since $M$ is hyperbolic.  Let $\gamma_1, \dots, \gamma_k$ be the pants decomposition of $N$ that we obtain above.

If $\gamma_i$ intersects one of the regions $U$ that we modified from $M$, then the fact that $\gamma_i$ cannot be contractible implies
that it must not lie within the $10 C \sqrt{\textrm{Area}(M)}$-neighborhood of $\alpha_j$ for every $j$.  If it did, then its diameter would be
at least
$$ 10 C \sqrt{\textrm{Area}(M)} \geq 5 C ( 2 \sqrt{\textrm{Area}(M)}) \geq 5 C \sqrt{\textrm{Area(N)}},$$
which is a contradiction.  Thus, $\{ \gamma_i \}$ lie entirely in the portion of $N$ which agrees with $M$, and so we can consider them
as curves in $M$.  As such, the diameter of each is $\lesssim \sqrt{\textrm{Area}(M)}$; from Proposition~\ref*{prop:hyperbolic}, $\textrm{Area}(M) \approx g + n$, which yields the desired bound.  Furthermore, since they constitute a pants decomposition of $N$, when we consider the cusps of $M$, we observe that they constitute a pseudo pants decomposition of $M$.

%%%%%%%%%%%%%%%%%%%%%%%%%%%%%%%%%%%%%%%%%%%%%%%%%%%%%%%%%%%%%%%%%
%%%%%%%%%%%%%%%%%%%%%%%%%%%%%%%%%%%%%%%%%%%%%%%%%%%%%%%%%%%%%%%%%
\section{Geodesic pseudo pants decompositions and the curve shortening process}
\label{sec:short}
%%%%%%%%%%%%%%%%%%%%%%%%%%%%%%%%%%%%%%%%%%%%%%%%%%%%%%%%%%%%%%%%%
%%%%%%%%%%%%%%%%%%%%%%%%%%%%%%%%%%%%%%%%%%%%%%%%%%%%%%%%%%%%%%%%%

In this section, we prove Proposition~\ref*{prop:pants_geodesics}.  The idea is to
employ the disk curve shortening process developed by Hass and Scott in \cite{HS}.
This process is defined for closed Riemannian surfaces; we will deal with
the issue of cusps shortly to work around this.  The idea is to start with a
finite sequence of piecewise smooth closed curves $\alpha_1, \dots, \alpha_n$ on
the closed Riemannian surface $S$.  We then choose a finite sequence of closed convex discs $D_1, \dots, D_m$ on
the surface $S$ of radius at most $\rho < \textrm{inj rad}(S)$ in
general position (here $\textrm{inj rad}(S)$ is the injectivity radius of $S$).  Roughly speaking, the idea of Hass and Scott is to move through the sequence of discs.
For each disc $D_i$ in the sequence, we replace each segment of each $\alpha_j$ which passes through $D_i$ with the unique length minimizing geodesic between
the same endpoints; since we chose the radii of $D_i$ to be sufficiently small, these
arcs are unique and also lie in $D_i$.

After doing this for all discs, we obtain a new sequence of piecewise smooth curves
$\alpha_{1,2}, \alpha_{2,2}, \dots, \alpha_{n,2}$.  Hass and Scott observed the
following:
\begin{enumerate}
    \item The length of $\alpha_{i,2}$ is no larger than the length of $\alpha_i$.
    \item If $\{ \alpha_i \}$ are simple, then $\{ \alpha_{i,2} \}$ are simple.
    \item If $\{ \alpha_i \}$ are all disjoint, then $\{ \alpha_{i,2} \}$ are disjoint.
    \item $\alpha_i$ is homotopic to $\alpha_{i,2}$.
\end{enumerate}

The procedure involves repeating this operation with $\{ \alpha_{i,2} \}$ to form $\{ \alpha_{i,3} \}$.  Repeating this procedure,
we obtain a $n$ sequences of curves $\{ \alpha_{i,j} \}$ for $i \in \{1, \dots, n\}$ and $j \in \{1, 2, \dots \}$.  Scott and Hass
also proved that, for a fixed $i \in \{1, \dots, n \}$, $\{ \alpha_{i,j} \}$ converges to a closed geodesic.  Note that in their
procedure there is an upper bound on the number of smooth segments of all curves involved (depending on the number of such segments in
the original curves and the number of discs); hence we can take this convergence to be $C^\infty$ at every point which is smooth in the sequence.

If $\alpha_i$ is simple and
noncontractible, then all $\{ \alpha_{i,j} \}$ are simple and noncontractible, and the resulting closed geodesic is simple.  Let $\{ \alpha_i^* \}$ be the resulting
closed geodesics.  We also have that, as a result of the uniqueness of geodesics, if all original curves are homotopically distinct and noncontractible,
then $\{ \alpha_i^* \}$ are simple closed curves which are disjoint and noncontractible.

In \cite{HS}, Hass and Scott generalize this curvature shortening procedure to work on families of curves.  For this article, however,
we will just need the results described above.  We are now in the situation that we have a hyperbolic surface $M$ with genus $g$ and $n$ cusps,
and a pseudo pants decomposition $\gamma_1, \dots, \gamma_k$ so that each curve has ambient diameter $\leq C \sqrt{g + n}$.

We begin by choosing open punctured discs $X_1, \dots, X_n$ around the $n$ cusps so that, if $\alpha$ is a closed curve with is noncontractible relative to the cusps, and if it intersects $X_i$, then $\alpha$ has length at least $100 C \sqrt{g + n}$.   $M \setminus \{ \bigcup X_i \}$ is compact; we can find
a finite number of closed convex discs $D_1, \dots, D_m$ all of radius $\leq \frac{\eta}{100}$ on $M$ which are in general position, and which
cover $M \setminus \{ \bigcup X_i \}$.  We cannot choose $\eta$ to be the injectivity radius of $M$, since it is $0$ if there is at least one cusp.  Instead, we choose $\eta > 0$ to be
    $$ \inf_{x \in M \setminus \{ \bigcup X_i \}} \textrm{inj rad}_x(M), $$
where $\textrm{inj rad}_x(M)$ is the injectivity radius of $M$ at the point $x$.  This infimum is positive because
$M \setminus \{ \bigcup X_i \}$ is compact; this is also why we can cover it with \emph{finitely} many convex
closed discs with this radius bound.
Since the length of each $\gamma_i$ is $\leq C \sqrt{g + n}$, and since each is noncontractible relative to the cusps
(since they form a pseudo pants decomposition), $\gamma_i$ cannot intersect any $X_j$, and so lies entirely in the union of $D_1, \dots, D_m$.

Fix a disc $D_i$ and a segment $\beta$ of $\gamma_i$ which starts and ends on $\partial D_i$, and which lies entirely in $D_i$.  If we replace $\beta$
with $\tilde{\beta}$, the unique shortest geodesic joining the endpoints of $\beta$, then the resulting curve $\eta$ is no longer than $\gamma_i$ and is homotopic
to $\gamma_i$, and so still lies in the union of all $D_i$.  In addition, we have the following lemma, which implies that $\eta$ has diameter $\lesssim \sqrt{g + n}$.

\begin{lem}
    \label{lem:diam}
    $M$ is a hyperbolic surface, and $\gamma$ is a curve in $M$.  Suppose that $\gamma$ has diameter
    $\leq D$, and suppose that $x$ and $y$ are points on $M$ which are of distance $\leq \frac{\eta}{100}$ from each other, where $\eta$ is chosen above.
    If we delete the segment from $x$ to $y$ and replace it with the unique length minimizing geodesic from $x$ to $y$,
    then the resulting curve $\tilde{\gamma}$ also has diameter $\leq D$.
\end{lem}
\begin{proof}
    If $a$ and $b$ are on the geodesic segment added, they they are closer than $x$ and $y$, and so the distance between
    $a$ and $b$ is $\leq D$.  If $a$ and $b$ are both not on the geodesic segment added, then the distance between them is $\leq D$.
    Lastly, if $a$ is not on the segment and $b$ is on the segment, then the distance from $a$ to $x$ is $\leq D$, and the distance
    from $a$ to $y$ is $\leq D$.  If $a$ is on the segment and $b$ is not on the segment, then the argument works in the same manner as below,
    just with the labels ``$a$" and ``$b$" reversed.
    
    Since the surface is hyperbolic, there is a covering map $F : \mathbb{H} \rightarrow M$ which is a local isometry.  We can
    use $F$ to lift $a$ to $a' \in \mathbb{H}$, then we can consider the ball $B$ of radius $D$ in $\mathbb{H}$ around $a'$.
    We can lift $x$ to a point $x'$ in $B$, and we can also lift the point $y$ to a point $y'$ in $B$.  The geodesic segment between $x$ and $y$
    that we add lifts to the unique geodesic joining $x'$ and $y'$; this follows from the fact that the segment has length less than
    $\frac{\eta}{100}$.  Since balls in $\mathbb{H}$ are geodesically convex, that is, if two points are in the ball, then the unique
    geodesic joining them is also in the ball, this geodesic arc is contained in $B$.  Thus, its image under $F$,
    which corresponds to the curve which $b$ is on, is also within a distance $D$ of $a$.  This completes the proof.
\end{proof}

To summarize, we move through the list of discs, then for each disc, we move through the list
of curves, then for each curve, we move through the list of arcs that pass through the disc, and then for each arc we perform the relevant replacement.
We can apply the curve shortening process of Hass and Scott to $\gamma_1, \dots, \gamma_k$ with respect to discs $D_1, \dots, D_m$; 
this forms curves $\{ \gamma_{i,2} \}$.  By continuing the apply their procedure (with $D_1, \dots, D_m$ fixed at the outset),
we obtain curves $\{ \gamma_{i,j} \}$ with $i \in \{1, \dots, k\}$ and $j \in \{1, 2, \dots \}$.  Furthermore, since the original curves were
simple and disjoint, all curves are simple and disjoint, and converge to closed geodesics $\gamma_1^*, \dots, \gamma_k^*$.

Since all $\gamma_1^*, \dots, \gamma_k^*$ are in different homotopy classes (since they form a pseudo pants decomposition), they are all disjoint
(this also uses the uniqueness of geodesics).  Hence, they also form a pseudo pants decomposition, are simple closed curves, and have the desired diameter
bound.  This completes the proof.

\bibliographystyle{amsplain}
\bibliography{bibliography}

\end{document}

%% file: Preamble.tex
\usepackage{amsmath,amssymb,amsthm,amsfonts,enumitem}
\usepackage{color}
\usepackage{graphicx,subcaption }
\usepackage{hyperref}
\usepackage{epstopdf}
\usepackage{bm}
\usepackage{mathtools}
\usepackage{tikz-cd}

\usepackage{pst-node}
\usepackage{auto-pst-pdf}
\usepackage{tikz-cd} 

%% file: Definitions.tex
\theoremstyle{plain} %default
\newtheorem{thm}{Theorem}
\newtheorem{lem}[thm]{Lemma}
\newtheorem{prop}[thm]{Proposition}

\theoremstyle{definition}
\newtheorem{defn}{Definition}
\newtheorem{conj}{Conjecture}

\theoremstyle{remark}